\pdfoutput=1
\documentclass[letterpaper, 10 pt, conference]{ieeeconf}  

\IEEEoverridecommandlockouts                              

\usepackage{cite}
\usepackage{color}
\usepackage{mathrsfs}
\usepackage{subcaption}
\usepackage{verbatim}

\usepackage{graphicx}
\usepackage{bbold}

\usepackage{tabularx}

\usepackage[cmex10]{amsmath}
\usepackage{array}
\usepackage{algorithmic}
\usepackage{amsfonts}
\usepackage[ruled, vlined]{algorithm2e}
\usepackage{url}
\usepackage{enumerate}
\usepackage{amsthm}
\usepackage{subcaption}   
\renewcommand{\ell}{{o_{j,d}}}

\theoremstyle{plain}
\newtheorem{coro}{Corollary}[section]
\newtheorem{thm}{Theorem}[section]
\newtheorem{lem}{Lemma}[section] 
\theoremstyle{definition}
\newtheorem{defn}[thm]{Definition} 
\title{\LARGE \bf
Passive-Aggressive Learning and Control
}

\author{Dimitar Ho, Nikolai Matni and John C. Doyle
\thanks{The authors are with the department of Control and Dynamical Systems, California Institute of Technology, Pasadena, CA 91125, USA ({\tt\small \{dho,nmatni,doyle\}@caltech.edu}).}%
\thanks{Thanks to funding from AFOSR and NSF and gifts from Huawei and Google.}
}

\begin{document}
\maketitle
\thispagestyle{empty}
\pagestyle{empty}


\begin{abstract}
In this work, we investigate the problem of simultaneously learning and controlling a system subject to adversarial choices of disturbances and system parameters. We study the problem for a scalar system with $l_\infty$-norm bounded disturbances and  system parameters constrained to lie in a known bounded convex polytope. We present a controller that is globally stabilizing and gives continuously improving bounds on the worst case state deviation. The proposed controller simultaneously learns the system parameters and controls the system. The controller emerges naturally from an optimization problem, and balances exploration and exploitation in such a way that it is able to efficiently stabilize unstable and adversarial system dynamics. Specifically if the controller is faced with large uncertainty, the initial focus is on exploration, retrieving information about the system by applying state-feedback controllers with varying gains and signs. In a pre-specified bounded region around the origin, our control strategy can be seen as \textit{passive} in the sense that it learns very little information.  Only once the noise and/or system parameters act in an adversarial way, leading to the the state exiting the aforementioned region for more than one time-step, our proposed controller behaves \textit{aggressively} in that it is guaranteed to learn enough about the system to subsequently robustly stabilize it. We end by demonstrating the efficiency of our methods via numerical simulations.
\end{abstract}

\section{Introduction}
With the proliferation of big-data, and the success of machine-learning algorithms being applied to planning and control problems, there has been a renewed interest in combining and applying learning and control to continuous systems.  Modern results build on the foundational ideas of adaptive control \cite{AstolfiAdaptive,AstromAdaptive}, which we cannot hope to adequately survey here, but place an emphasis on finite-time, rather than asymptotic, guarantees of performance and stability.

To the best of our knowledge, recent results of this nature have focused on the stochastic setting wherein system parameters are unknown, and must be identified despite stochastic excitations to the system.  By combining concentration results from high-dimensional statistics with techniques from robust and optimal control, regret and performance bounds can be obtained as a function of the number of data points seen by the controller. Notable examples include \cite{AndersCDC,LearningLQR,Fietcher,Abbasi1,Abbasi2}, which provide varying degrees of guarantees and practically applicable algorithms.  However, as far as we are aware, no comparable results exist for the setting of bounded but adversarial process noise and parametric uncertainty.

Recently it has been shown that in the $l_\infty$ bounded adversarial setting, solutions to the state-estimation problem \cite{Nair1,Nair2} and the robust control problem subject to quantization and delay in the control loop \cite{Yorie} admit particularly intuitive and appealing forms.  This work shows that the same holds true for a joint learning and control problem.

Our main contribution is what we call a passive-aggressive learning and control algorithm that trades off between identifying the true system parameters and stabilizing the system.  The defining feature of this controller is that unless the system parameters and noise act in an adversarial way, pushing the state sufficiently far away from the origin, it is content with passively observing the state evolution and updating its uncertainty set.  However, when the system conspires to push the state sufficiently far from the origin, it aggressively learns the system parameters and applies control actions aimed at stabilizing the system.

The rest of the paper is organized as follows: in Section II we define the problem and the necessary notions of consistent parameter sets given a sequence of observations. In Section III we then show that in the case of ``strongly stabilizable'' initial parameter uncertainty sets, a simple static state-feedback policy is sufficient to guarantee robust stability for all possible choices of system realization.  We then build on this result in Section IV to show that if the controller updates the set of feasible system parameters with each observation, the controller performance can be strictly improved.  Finally, in Section V we consider the case of general initial uncertainty sets, and show that if a two-stage controller is applied, then the uncertainty set can eventually be reduced to one that is strongly stabilizable, allowing us to switch to the aforementioned control policies.  We demonstrate the efficacy of our approach in Section VI, and end with conclusions and future work in Section VII.

\section{System and Problem Definition}
\subsection{System Dynamics}
\noindent We consider the scalar linear discrete-time system 
\begin{align}
\label{eq:dyn} x_{n+1} &= a x_n + b u_n(x_{0:n}) + w_n\\
&\left|w_n\right| \leq \eta\,\,\,\begin{bmatrix}a\\b \end{bmatrix} \in \mathcal{P}_0
\end{align}
with the state $x_n$, the disturbance $w_n$ and the causal controllers $u_n(x_{0:n})$ where we use the notation $x_{0:n}$ to refer to the stacked vector $\left[x_0,x_1,\dots, x_n \right]^\mathrm{T}$. We assume that the disturbance $w_n$ is $l_\infty$ bounded by $\eta$, and that the state-space parameters $a$ and $b$ are unknown constants, but that are constrained to lie in a known bounded convex polytope $\mathcal{P}_0$. 

We furthermore assume that
\begin{align}
b\neq 0\,\, \forall \begin{bmatrix}a\\b \end{bmatrix} \in \mathcal{P}_0,
\end{align}
i.e., that the system is controllable for all possible realizations of the unknown state-space parameters $(a,b)$. To simplify notation, we will refer to the controllers as $u_n$, where the subscript reminds that $u_n$ is a function of $x_{0:n}$.

\subsection{Consistent Sets}
We denote by $x_{i:j}$ and $u_{i:j}$ the stacked vector of state values $x_n$ and control inputs $u_n$, respectively, for $i\leq n\leq j$.  It follows from the dynamics \eqref{eq:dyn} that at time $N$, the following entry-wise inequality must hold for the true parameters $(a,b)$.
\begin{align}
\left| x_{1:N}-\left[x_{0:N-1},u_{0:N-1}\right]\begin{bmatrix}a\\b \end{bmatrix}\right| \leq \mathbb{1} \eta,
\end{align}
where $\mathbb{1}$ is the all ones vector of compatible dimension.  This inequality therefore allows us to characterize the subset of the initial uncertainty set $\mathcal{P}_0$ that is consistent with the observed state and control input histories given the known bound $\eta$ on the magnitude of the disturbance process $w_n$.

This motivates the following definition of the \emph{consistent set} at time $N$:
\begin{align}
\notag&\mathcal{S}(x_{0:N},u_{0:N-1})\\
:=& \left\{\left. \begin{bmatrix}a\\b\end{bmatrix} \right|\left| x_{1:N}-\left[x_{0:N-1},u_{0:N-1}\right]\begin{bmatrix}a\\b \end{bmatrix}\right| \leq \mathbb{1} \eta\right\}.
\end{align}

It then follows that given state and control histories $x_{0:N}$, $u_{0:N-1}$ and the initial uncertainty set $P_0$, we have that $[a,b]^T$ lies in the bounded convex polytope $P_0 \cap \mathcal{S}(x_{0:N},u_{0:N-1})$.

For $N=1$, the set $\mathcal{S}(x_{0:1},u_{0})$ reduces to a slice of thickness $2\eta/\sqrt{x^2_0 +u^2_0}$ with normal vectors $\pm[x_0,u_0]^T$ in parameter space (see Fig.\ref{img:slice}). This motivates the following recursive definition of the consistent set 
$\mathcal{S}(x_{0:N},u_{0:N-1})$ at time $N$ as the intersection of $N$ such slices:
\begin{align}
\mathcal{S}(x_{0:N},u_{0:N-1}) = \bigcap\limits^{N-1}_{i=0}  \mathcal{S}(x_{i:i+1},u_{i}).
\end{align}

\begin{figure}[h]
\centering
\includegraphics[width=0.4\textwidth]{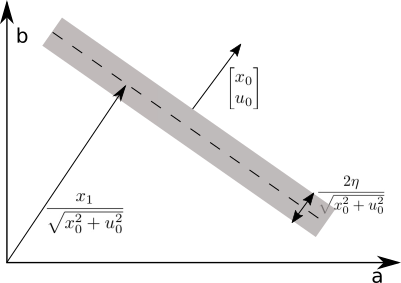}
\caption{Example of $S(x_{0:1},u_0)$}
\label{img:slice}
\end{figure}

\subsection{Problem Statement}\label{sec:prob}
Our objective is to find the best causal control strategy $u_k(x_{0:k},\mathcal{P}_0)$ that minimizes the worst-case state deviation $\left\|x_{1:\infty} \right\|_{\infty}$ despite  adversarial noise $w_{0:\infty}$ and system parameter choices $[a,b] \in \mathcal{P}_0$.

Formally, we seek a solution to the following infinite-horizon min-max problem 
\begin{equation}
V\left(x_0,\mathcal{P}_0\right)
\begin{array}{rl}\\
:= \min\limits_{u_{0:\infty}} Q^{1:\infty}(x_0,\mathcal{P}_0,u_{0:\infty})
\end{array}
\label{eq:problem}
\end{equation}
where we define $Q^{1:N}(x_0,\mathcal{P}_0,u_{0:{N-1}})$ as
\begin{align}
\notag &Q^{1:N}(x_0,\mathcal{P}_0,u_{0:{N-1}})\\
:=& \begin{array}{rl}
\max\limits_{\left[\begin{smallmatrix} a\\b\end{smallmatrix}\right]\in \mathcal{P}_0}  \max\limits_{\left\|w_{0:N-1}\right\|_{\infty}\leq \eta} & \left\| x_{1:N} \right\|_{\infty} \\
\text{s.t.} & \text{dynamics \eqref{eq:dyn}.} 
\end{array}\\
\label{eq:QR}= &\max\limits_{x_{1:N} \in \mathcal{R}^N_\eta(x_0,u_{0:N},\mathcal{P}_0)} \left\|x_{1:N} \right\|_{\infty},
\end{align}
where \eqref{eq:QR} is an equivalent reformulation using the following definition.
\begin{defn}\label{def:Rn}
Define the N-step reachable set $\mathcal{R}^N_{\eta}(x_0,u_{0:\infty},\mathcal{P}_0)$ as the set of possible trajectories $x_{1:N}$, given the initial condition $x_0$, the initial uncertainty set $\mathcal{P}_0$, the controllers $u_{0:\infty}$ and the disturbance bound $\eta$, i.e.
\begin{align}
&\mathcal{R}^N_\eta(x_0,u_{0:N},\mathcal{P}_0)\\
 =& \left\{ v_{1:N}\left| \begin{array}{c} \exists \begin{bmatrix} a\\b\end{bmatrix} \in \mathcal{P}_0, \exists |w_n| \leq \eta, s.t\\
  v_0 = x_0,\,\,v_{n+1} = av_{n}+bu_{n}(v_{0:n}) + w_n \end{array} \right.\right\}
\end{align}
\end{defn}

Our approach is to begin with what we term \emph{strongly stabilizable} initial uncertainty sets $\mathcal{P}_0$, that is to say initial uncertainty sets for which an appropriately chosen static state-feedback gain is guaranteed to be stabilizing for all realizations of the system parameters $(a,b)$.  We show that such a policy is optimal for optimization problem \eqref{eq:problem} restricted to static memoryless control policies. We then show that by adding an adaptive element to such a control policy, the performance can be further improved, and that an exploration/exploitation strategy naturally emerges.  Finally, we tackle the case of general initial uncertainty sets and show that after an initial ``passive'' learning phase, the uncertainty set is eventually ``aggressively'' reduced to one that is strongly stabilizable, allowing for our previously derived adaptive strategy to be applied.


\section{Robust Static State-Feedback for Strongly Stabilizable Uncertainty Sets}

In this section, we consider a restriction of optimization problem \eqref{eq:problem} to static state-feedback control policies, i.e., we restrict $u_n = kx_n$ for all $n\geq 0$.  The resulting optimization problem then reads as
\begin{align}
\noindent &V_{RSF}\left(x_0,\mathcal{P}_0\right) =\begin{array}{rl} \min\limits_{k} & Q^{1:\infty}\left(x_0,\mathcal{P}_0,u_{0:\infty}\right) \\
\text{s.t.} & u_n = kx_n,\forall n\in \mathbb{N}\end{array}
\label{eq:sf_problem}
\end{align}


As we are restricting ourselves to static state-feedback polices, it follows that $V_{RSF}(x_0,\mathcal{P}_0)$ is an upper bound for  $V(x_0,\mathcal{P}_0)$, i.e.:
\begin{align}
 V(x_0,\mathcal{P}_0) \leq V_{RSF}(x_0,\mathcal{P}_0).
\end{align}

We consider this simpler problem as it has several appealing properties.  First,  $V_{RSF}\left(x_0,\mathcal{P}_0\right)$ and the corresponding minimizing $k^*$ can be solved for in closed form.  Further it motivates the definition of \emph{strongly stabilizable} initial uncertainty sets $\mathcal{P}_0$, which naturally captures how difficult an uncertain system can be to stabilize.  To that end, we introduce the following measure of stabilizability for an uncertainty set $\mathcal{P}$.
\begin{defn}\label{def:lamb}
We use $\lambda(\mathcal{P})$ to denote the \textit{stability margin} of the parameter set $\mathcal{P}$, and define it as
\begin{align}
\lambda(\mathcal{P}) &:= \min_{k} \max_{\left[\begin{smallmatrix} a\\b\end{smallmatrix}\right]\in \mathcal{P}} \left|a+bk\right|.
\end{align}
Furthermore, we call the corresponding minimizer
\begin{align}
K(\mathcal{P}) = \text{argmin}_{k} \max_{\left[\begin{smallmatrix} a\\b\end{smallmatrix}\right]\in \mathcal{P}} \left|a+bk\right|
\end{align}
the \textit{gain} of the parameter set $\mathcal{P}$.
\end{defn}

The stability margin of a set $\lambda(\mathcal{P})$ is a functional mapping sets to $\mathbb{R}^{+}_{0}$ and describes the smallest system eigenvalue achievable by a constant state-feedback, assuming worst case parameter choice of $\left[a,b \right]^T \in \mathcal{P}$. As presented in the following Lemma, it follows that if the initial uncertainty set $\mathcal{P}_0$ satisfies
$\lambda(\mathcal{P}_0) <1$, then we can use $k = K(\mathcal{P}_0)$ as a state-feedback control law to stabilize the system for all parameters in $\mathcal{P}_0$.  We will refer to such initial uncertainty sets $\mathcal{P}_0$ as \textit{strongly stabilizable}.

In the following we let $a \vee b:=\max\left\{a, b\right\}$ to help simplify notation.
\begin{lem}[Comparison-Lemma] \label{lem:comp}
Let $v_n$ be a non-negative sequence in $\mathbb{R}$ that satisfies $v_{n+1} \leq f_n(v_{n})$ with non-decreasing functions $f_n$, then $v_n$ is bounded above by the sequence $\gamma_n$, where $\gamma_0 = v_0$ and $\gamma_{n+1} = f_n(\gamma_n)$.
\end{lem}
\begin{proof}
The result follows by induction from $\gamma_0 \leq v_0$ and $v_{n+1} \leq f_n(v_{n}) \leq f_n(\gamma_n) = \gamma_{n+1}$.
\end{proof}

\begin{lem}\label{lem:RSF} If the initial uncertainty set $\mathcal{P}_0$ is strongly stabilizable (i.e., if $\lambda(\mathcal{P}_0)<1$), then the problem \eqref{eq:sf_problem} attains its optimum with the controller $u^{RSF}_n = K(\mathcal{P}_0)x_n$ and 
\begin{equation}
V_{RSF}(x_0,\mathcal{P}_0)  =\left(\lambda(\mathcal{P}_0)|x_0| + \eta \right)\vee\left( \frac{\eta}{1-\lambda(\mathcal{P}_0)}\right).
\label{eq:sf_cost}
\end{equation}
Conversely, if the initial uncertainty set $\mathcal{P}_0$ is not strongly stabilizable (i.e., if $\lambda(\mathcal{P}_0)\geq 1$), then for any choice of $k$, it holds that $V_{RSF}(x_0,\mathcal{P}_0) = \infty$.
\end{lem}

\begin{proof}
Fix $u_n = kx_n$ for some $k$ and consider the case $\lambda(\mathcal{P}_0) \geq 1$. By definition \eqref{def:lamb}, $\exists \left[a^*,b^*\right]^\mathrm{T} \in \mathcal{P}_0$ such that $\left| a^*+b^*k\right| \geq 1$ and it is easy to see that $\forall x_0 \in \mathbb{R}$ we can find $w_{0:\infty}$ to make $x_n$ grow unbounded. Therefore, we conclude $V_{RSF}(x_0,\mathcal{P}_0) = \infty$ if $\lambda( \mathcal{P}_0) \geq 1$.\\
For the second part of the proof, notice that $\left|x_{n+1}\right| \leq \left|(a+bk)\right|\left|x_n\right| + \eta$
and by the comparison lemma \eqref{lem:comp}, we obtain the bound 
\begin{align}
\left|x_n \right| \leq \gamma_{n} = \left|a+bk\right|^n\left|x_0\right| + \frac{\left|a+bk\right|^n-1}{\left|a+bk\right|-1}\eta
\end{align} 
Now notice that $\gamma_n$ is monotonic and if $\lambda(\mathcal{P}_0) < 1$ we have that $\left\|\gamma_{1:\infty}\right\|_{\infty}$ is bounded and it takes either the value $|\gamma_1|$ or $\lim_{n\rightarrow \infty}|\gamma_n|$. Therefore
\begin{align}\label{eq:gamrsf}
\left\|\gamma_{1:\infty}\right\|_{\infty} = \left(\left|a+bk\right||x_0| + \eta \right)\vee\left( \frac{\eta}{1-\left|a+bk\right|}\right)
\end{align} 
Furthermore, notice that for each $n$ and fix $a, b$ and $k$, we can construct a sequence $w_{0:\infty}$ such that $|x_n| = \gamma_n$. Hence, we can rewrite the optimization problem \eqref{eq:sf_problem} equivalently as
\begin{align*}
&V_{RSF}(x_0,\mathcal{P}_0)\\
= &\min\limits_{k} \max\limits_{\left[\begin{smallmatrix} a\\b\end{smallmatrix}\right]\in \mathcal{P}_0} \left(\left|a+bk\right||x_0| + \eta \right)\vee\left( \frac{\eta}{1-\left|a+bk\right|}\right)
\end{align*}
We can now conclude the desired statement \eqref{eq:sf_cost}, by noticing that expression \eqref{eq:gamrsf} is monotonic in $|a+bk|$ and therefore we obtain the optimum value at $k^* = K(\mathcal{P}_0)$.
\end{proof}
Following the arguments of the proof of Lemma \eqref{lem:RSF}, we also obtain the following performance bound of the \textit{robust state-feedback} (RSF) controller $u_n = K(\mathcal{P}_0)x_n$:
\begin{coro}[RSF bound]
Assume the dynamics \eqref{eq:dyn} with $\lambda\left( \mathcal{P}_0\right)<1$ and $u_n = K\left( \mathcal{P}_0\right)x_n$, then $\left| x_n\right| \leq \gamma^{RSF}_n$, where 
\begin{align}
\gamma^{RSF}_n  =  \lambda\left( \mathcal{P}_0\right)^n\left|x_0\right| + \frac{\lambda\left( \mathcal{P}_0\right)^n-1}{\lambda\left( \mathcal{P}_0\right)-1}\eta 
\end{align}
\end{coro}

Although the RSF controller is guaranteed to stabilize systems with strongly stabilizable uncertainty sets $\mathcal{P}_0$ and  \eqref{eq:sf_cost} provides a finite upper bound to $V(x_0,\mathcal{P}_0)$, it does so in an inefficient way. In particular, it does not update its control policy to reflect the fact that with each observation, more information about the underlying true parameters is revealed. As previously discussed, the observations $x_{0:N}, u_{0:N-1}$ allow us to reduce the space of consistent parameters $[a,b]^T$ to be $P_0\cap S(x_{0:N},u_{0:N-1})$. 
In what follows we improve upon the static state-feedback policy results of this section, and ultimately show that learning is necessary to compute stabilizing controllers for general initial uncertainty sets $\mathcal{P}_0$.

\section{Robust Adaptive State-Feedback for Strongly Stabilizable Uncertainty Sets}\label{sec:ARSF}

Keeping our focus on strongly stabilizable initial uncertainty sets, we propose two controllers that strictly outperform the static state-feedback policy $k=K(\mathcal{P}_0)$ defined in the previous section. The following adaptive schemes, which we call \textit{weakly adaptive RSF} (WRSF) and \textit{strongly adaptive RSF} (SRSF), simultaneously learn the system dynamics while controlling the system. The latter algorithm decides at every time-step $n$ between a control action that reduces $\left|x_{n+1} \right|$ and an exploratory control action that leads to more information about the system parameters. Our key result is a decomposition theorem that exploits the fact that the control policy at time $n$ is allowed to be a function of all past state-measurements $x_{0:n}$.  
\subsection{Decomposition and Properties of $V(x_0,\mathcal{P}_0)$}
Using the definition \eqref{eq:QR} and Lemma \eqref{lem:Rn}, it is easy to derive the following properties of $V(x_0,\mathcal{P}_0)$:
\begin{lem}\label{lem:Rn}
$\forall \alpha >0$ holds 
\begin{align}
\mathcal{R}^N_\eta(\alpha x_0,\left\{u_n\right\},\mathcal{P}_0) = \alpha \mathcal{R}^N_{\eta/\alpha}( x_0,\left\{1/\alpha u_n(\alpha x_{0:N})\right\},\mathcal{P}_0)
\end{align}
\end{lem}
\begin{proof}
This follows by change of variables and using definition \eqref{def:Rn}.
\end{proof}
\begin{coro}\label{cor:Vsym}
$\forall x_0$ holds $V(x_0,\mathcal{P}_0) = V(-x_0,\mathcal{P}_0)$
\end{coro}

We then have that the following decomposition theorem holds:
\begin{thm}\label{thm:minmax}
For the cost-to-go function $V$ as defined in optimization problem \eqref{eq:problem}, it holds that
\begin{align*}
&V(x_0,\mathcal{P}_0)\\
=&\min\limits_{u_0} \left[Q^{1:1}(x_0,\mathcal{P}_0, u_0) \vee \right. \\ & \left. \max\limits_{x_1 \in \mathcal{R}^1_{\eta}(x_0,\mathcal{P}_0,u_0)} V\left(x_1,\mathcal{P}_0\cap \mathcal{S}(x_{0:1},u_0)\right)\right]
\end{align*}
\end{thm}
\begin{proof}
First notice that we can write
\begin{align*}
 &V\left(x_0,\mathcal{P}_0\right)\\
=&\min\limits_{u_{0:\infty}} Q^{1:\infty}(x_0,\mathcal{P}_0,u_{0:\infty})\\
=&\min_{u_{0:\infty}}\max\limits_{\left[\begin{smallmatrix} a\\b\end{smallmatrix}\right]\in \mathcal{P}_0}  \max\limits_{\left\|w_{0:\infty}\right\|_{\infty}\leq \eta}  \left\| x_{1:\infty} \right\|_{\infty} \text{ s.t. dynamics \eqref{eq:dyn}} \\
=& \min_{u_0}\left[\max_{x_1 \in \mathcal{R}^1_{\eta}(x_0,u_0,\mathcal{P}_0)} \right. \min_{u_{1:\infty}}  \\
& \left.
\max\limits_{\left[\begin{smallmatrix} a\\b\end{smallmatrix}\right]\in \mathcal{P}_0\cap \mathcal{S}(x_{0:1},u_0)}\max\limits_{\left\|w_{1:\infty}\right\|_{\infty}\leq \eta} |x_1| \vee \left\| x_{2:\infty} \right\|_{\infty}\right]\\
\notag&\text{ s.t. dynamics \eqref{eq:dyn}},
\end{align*}
where the last equality follows from the fact that the state $x_1$ is known to the sequence of control actions $u_{1:\infty}$, and that $\|x_{1:\infty}\|_\infty = |x_1|\vee \|x_{2:\infty}\|_\infty$.  This last line can further be rewritten as
\begin{align*}
&V\left(x_0,\mathcal{P}_0\right)\\
=& \min_{u_0}\left[\left(\max_{x_1 \in \mathcal{R}^1_{\eta}(x_0,u_0,\mathcal{P}_0)}|x_1|\right) \vee  \left(\max_{x_1 \in \mathcal{R}^1_{\eta}(x_0,u_0,\mathcal{P}_0)}\right.\right. \\
& \left.\left. \dots\min_{u_{1:\infty}}\max\limits_{\left[\begin{smallmatrix} a\\b\end{smallmatrix}\right]\in \mathcal{P}_0\cap \mathcal{S}(x_{0:1},u_0)}\max\limits_{\left\|w_{1:\infty}\right\|_{\infty}\leq \eta} \|x_{2:\infty\|_\infty} \right)\right] \\ & \text{s.t. dynamics \eqref{eq:dyn}} \\
& = \min_{u_0} \left[Q^{1:1}(x_0,\mathcal{P}_0,u_0)\vee\right. \\ & \left.\max_{x_1 \in \mathcal{R}^1_{\eta}(x_0,u_0,\mathcal{P}_0)} V(x_1, \mathcal{P}_0\cap \mathcal{S}(x_{0:1},u_0),u_{1:\infty})\right]
\end{align*}
where the first equality follows from the fact that $\max_x f(x) \vee g(x) = (\max_x f(x))\vee(\max_x g(x))$, and the second from the definition of the cost-to-go function $V$, as defined in \eqref{eq:problem}, and from the identity $\max\limits_{x_1 \in \mathcal{R}^1_{\eta}(x_0,u_0,\mathcal{P}_0)} |x_1| = Q^{1:1}(x_0,\mathcal{P}_0, u_0(x_0))$.
\end{proof}

Theorem \ref{thm:minmax} sheds light on the structure of the optimal control policy. Specifically, let $\mathcal{P}_i :=\mathcal{P}_0 \cap \mathcal{S}(x_{0:i},u_{0:i-1})$; then the optimal control action at time $i$ is given by\footnote{This follows from Theorem \ref{thm:minmax} by a simple induction argument which is omitted in the interest of space.}
\begin{align}\label{eq:optctrl}
&u_i(x_{0:i})
= \mathrm{argmin}_{u} \max\limits_{x_{i+1} \in \mathcal{R}^1_{\eta}(x_i,\mathcal{P}_i,u)} |x_{i+1}| \vee \\ & \indent V\left(x_{i+1},\mathcal{P}_i\cap \mathcal{S}(x_{i:i+1},u)\right).
\end{align}

In particular, we see that the control action $u_i$ is a function of both the state history $x_{0:i}$ and the updated uncertainty set $\mathcal{P}_i$, and naturally results in a trade-off between exploration and exploitation. 
If the first term dominates the cost function, this indicates that the controller is in an exploitation mode, using its gathered information on the uncertainty set to minimize state deviation.  In contrast, if the second term dominates, this can be interpreted as an exploration action aimed at reducing the effects of parametric uncertainty on future state deviations.

\subsection{Weakly and Strongly Adaptive Robust State-Feedback Controller}
Unfortunately, Eq. \eqref{eq:optctrl} does not provide a practical means of computing an optimal controller. Nevertheless, we can approximate \eqref{eq:optctrl} by using $V_{RSF}$ as an upper bound for the cost-to-go function in \eqref{eq:optctrl}. In particular, we will define the \textit{weakly adaptive robust state-feedback} (WRSF) controller as the solution to the optimization problems
\begin{align}\label{eq:WRSFctrl}
&u^{WRSF}_n(x_n,\mathcal{P}_n)\\
\notag= &\text{argmin}_{u} \max\limits_{x_{n+1} \in \mathcal{R}^{1}_{\eta}(x_n,\mathcal{P}_n,u)}|x_{n+1}| \vee V_{RSF}\left(x_{n+1},\mathcal{P}_n\right)
\end{align}
and the \textit{strongly adaptive robust state-feedback} (SRSF) controller as the solution to the optimization problems
\begin{align}\label{eq:SRSFctrl}
&u^{SRSF}_n(x_n,\mathcal{P}_n)\\
\notag= &\text{argmin}_{u} \max\limits_{x_{n+1} \in \mathcal{R}^{1}_{\eta}(x_n,\mathcal{P}_n,u)}\\
&\dots|x_{n+1}| \vee V_{RSF}\left(x_{n+1},\mathcal{P}_n\cap \mathcal{S}(x_{n:n+1},u)\right)
\end{align}
Accordingly we will define $V_{WRSF}(x_n,\mathcal{P}_n)$ and $V_{SRSF}(x_n,\mathcal{P}_n)$ as 
\begin{align}\label{eq:VARSF}
&V_{WRSF}(x_n,\mathcal{P}_n)\\
\notag= &\min\limits_{u} \max\limits_{x_{n+1} \in \mathcal{R}^{1}_{\eta}(x_n,\mathcal{P}_n,u)}|x_{n+1}| \vee V_{RSF}\left(x_{n+1},\mathcal{P}_n\right)\\
&V_{SRSF}(x_n,\mathcal{P}_n)\\
\notag= &\min_{u} \max\limits_{x_{n+1} \in \mathcal{R}^{1}_{\eta}(x_n,\mathcal{P}_n,u)}\\
&\dots|x_{n+1}| \vee V_{RSF}\left(x_{n+1},\mathcal{P}_n\cap \mathcal{S}(x_{n:n+1},u)\right)
\end{align}
After some standard manipulation, we can formulate the resulting controllers as
\begin{align}\label{eq:arsfctrls}
u^{WRSF}_n &= K(\mathcal{P}_n) x_n\\
\notag u^{SRSF}_n
 &=\text{argmin}_{u} \max\limits_{x_{n+1} \in \mathcal{R}^{1}_{\eta}(x_n,\mathcal{P}_n,u)}\\
\label{eq:eqSRSF}&\dots|x_{n+1}| \vee \frac{\eta}{1- \lambda\left(\mathcal{P}_n\cap \mathcal{S}(x_{n:n+1},u)\right)}
\end{align}
$u^{WRSF}_n$ turns out to present a simple yet significantly better strategy than the RSF controller, as it is basically using the information of recent observations to recompute an RSF controller. Therefore it is not surprising that we have 
$V_{WRSF}(x_n,\mathcal{P}_n) = V_{RSF}(x_n,\mathcal{P}_n)$

To apply $u^{SRSF}$ on the other hand, requires the solution of a scalar min-max problem at every time-step.  Notice from \eqref{eq:eqSRSF}, how the objective of  SRSF policy can be seen as a exploration vs. exploitation trade-off, as it incorporates the reduced uncertainty of the next time step. As we will show in the following theorem, if the initial uncertainty set is strongly stabilizable, then this SRSF policy is stabilizing and performs at least as well as the WRSF and RSF policy.
\begin{thm}[SRSF vs. WRSF vs. RSF]\label{thm:SRSFRSF}
Let $x^{RSF}_n$, $x^{WRSF}_n$, $x^{SRSF}_n$ be sequences generated from running $u^{RSF}$, $u^{WRSF}$ and $u^{SRSF}$ in closed loop with the same initial condition $x_0$ and assume a strongly stabilizable $\mathcal{P}_0$ with the usual dynamics \eqref{eq:dyn}. Furthermore, assume that all sequences obtain the same uncertainty sets $\mathcal{P}_{n}$, then
$x^{RSF}_n$, $x^{SRSF}_n$ and $x^{WRSF}_n$ are upper bounded by
\begin{align}
\label{eq:BSRSF}|x^{SRSF}_n| &\leq \gamma^{SRSF}_n \\
\label{eq:BWRSF}|x^{WRSF}_n| &\leq \gamma^{WRSF}_n\\
\label{eq:RSF}|x^{RSF}_n| &\leq \gamma^{RSF}_n 
\end{align}
where 
\begin{align}
\gamma^{SRSF}_{n+1} &= V_{SRSF}(\gamma^{SRSF}_n,\mathcal{P}_n) \\
\gamma^{WRSF}_{n+1} &= V_{RSF}(\gamma^{WRSF}_n,\mathcal{P}_n)\\
\gamma^{RSF}_{n+1} &= V_{RSF}(\gamma^{RSF}_n,\mathcal{P}_0)
\end{align}
with $\gamma^{SRSF}_0 = \gamma^{WRSF}_0 = \gamma^{RSF}_0 = |x_0|$.
Furthermore, we have $\gamma^{SRSF}_n \leq \gamma^{WRSF}_n \leq \gamma^{RSF}_n$ 
\end{thm}
\begin{proof}
While we have proven the first part already for the RSF controller, we can obtain the performance bounds for WRSF and SRSF, by observing from their definition that they are upperbounds on the worst-case adversarial dynamics. Therefore, at every-time step $n$ holds
\begin{align}
|x^{SRSF}_{n+1}| &\leq V_{SRSF}(x^{SRSF}_{n},\mathcal{P}_n) \\
|x^{WRSF}_{n+1}| & \leq V_{WRSF}(x^{WRSF}_{n},\mathcal{P}_n) 
\end{align}
Then using the comparison lemma \eqref{lem:comp}, we can establish
$|x^{SRSF}_n| \leq \gamma^{SRSF}_n$ and $|x^{WRSF}_n| \leq \gamma^{WRSF}_n$. The inequality $\gamma^{SRSF}_n \leq \gamma^{WRSF}_n \leq \gamma^{RSF}_n$ follows, by observing that 
\begin{align*}
V_{SRSF}(x,\mathcal{P}_n) \leq V_{RSF}(x,\mathcal{P}_n) \leq V_{RSF}(x,\mathcal{P}_0)
\end{align*} 
and noticing that $V_{WRSF}$ and $V_{RSF}$ are monotonic increasing in $|x|$ and by using comparison Lemma \eqref{lem:comp}.
\end{proof}

\section{Passive Aggressive Feedback Controller}

In this section we introduce a control policy that is applicable to initial uncertainty sets that are not strongly stabilizable.  The controller evolves according to two stages: at first, a ``passive-aggressive'' feedback controller is deployed that is used as long as the consistent set $\mathcal{P}_0 \cap \mathcal{S}(x_{0:N},u_{0:N-1})$ is not strongly stabilizable.  We show that this control policy is guaranteed to shrink the initial uncertainty set to one that is strongly stabilizable once the state becomes sufficiently large. Once the uncertainty set has been reduced to a strongly stabilizable one, the controller switches to the ARSF strategy described in the previous section and drives the state to the origin.

\begin{defn}
Let $\mathcal{P}_n$ be the remaining uncertainty in the system parameters after observing $x_{0:n}$ and $u_{0:n-1}$. Specifically
\begin{align}
\mathcal{P}_{n+1} &= \mathcal{P}_{n}\cap \mathcal{S}\left(x_{n:n+1},u_n \right)
\end{align}
where we set $\mathcal{P}_{0}$ to be the initial  uncertainty set.
\end{defn}
\begin{defn}
Define $k_{max}(\mathcal{P}_n)$ as the maximum deadbeat controller gain among the parameters in $\mathcal{P}_n$:
\begin{align}
k_{max}(\mathcal{P}_n) := \max\limits_{\left[\begin{smallmatrix} a\\b\end{smallmatrix}\right]\in \mathcal{P}_n} \left| -\frac{a}{b}\right| 
\end{align}
\end{defn}

We begin with an intermediate result that shows if the system state is sufficiently large, then the stability margin of the uncertainty set can be reduced by an amount governed by the noise bound $\eta$ and the size of the state itself.  In this way, there is a notion of signal-to-noise that comes into play in the ability to learn an uncertainty set.

\begin{thm}[Passive-Aggressive Learning]\label{thm:PALearn}
Let $\mathcal{P}_0$ be the initial (not necessarily strongly stabilizable) uncertainty set and fix positive constants $\lambda^*,p>0$ satisfying $\lambda^*>\frac{1}{p}$. Consider the system \eqref{eq:dyn} with the following time-varying state-feedback controller:
\begin{align}\label{eq:learnu}
u^{PAL}_n (x_n)&=k_n x_n 
\end{align}
where
\begin{align*}
k_n = (-1)^n\frac{k_{max}(\mathcal{P}_n)}{\lambda^*p-1}
\end{align*}
Then $\forall n$ that satisfy $\min\{|x_{n-1}|,|x_{n-2}|\} \geq p\eta$ holds $\lambda(\mathcal{P}_{n}) \leq \lambda^*$.
\end{thm}
\begin{proof} Assume $n_0$ to be such that $\min\{|x_{n_0-1}|,|x_{n_0-2}|\} \geq p\eta$ and $n_{-1}$, $n_{-2}$ to be $n_0-1$ and $n_0-2$ respectively. Then,
\begin{align}
\notag\lambda(\mathcal{P}_{n_0}) &= \lambda\left(\mathcal{P}_{n_{-1}} \cap \mathcal{S}(x_{n_{-2}:n_0},u_{n_{-2}:n_{-1}}) \right)\\
\notag&\leq \lambda\left(\mathcal{P}_{n_{-1}} \cap \mathcal{B}(x_{n_{-2}:n_0},u_{n_{-2}:n_{-1}}) \right)\\
\notag&\leq k_{max}\left(\mathcal{P}_{n_{-1}}\right)\Delta_{b}\mathcal{B}(x_{n_{-2}:n_0},u_{n_{-2}:n_{-1}}) + \\
\label{eq:delbdela}&\dots \Delta_{a}\mathcal{B}(x_{n_{-2}:n_0},u_{n_{-2}:n_{-1}}) 
\end{align}
where $\mathcal{B}(\dots)$ represents the smallest outer-bounding box set of the $\mathcal{S}(x_{n_{-2}:n_0},u_{n_{-2}:n_{-1}})$. Furthermore, $\Delta_b \mathcal{B}$ and $\Delta_a \mathcal{B}$ are the maximum uncertainty of parameters $a$ and $b$ in the set $\mathcal{B}$ as discussed in the appendix \eqref{sec:box}. Finally we note that the last inequality follows from the discussion on stability margins of boxed uncertainties in \eqref{sec:box}. Now, as $k_{n_{-2}}$ and $k_{n_{-1}}$ have opposite sign by construction, we obtain from app.\eqref{sec:box}:
\begin{align}
\notag&\Delta_{b}\mathcal{B}(x_{n_{-2}:n_0},u_{n_{-2}:n_{-1}})\\
\label{eq:delb}=&\left(\frac{\eta}{|x_{n_{-2}}|} + \frac{\eta}{|x_{n_{-1}}|}\right)\frac{1}{|k_{n_{-1}}| + |k_{n_{-2}}|}\\
\notag&\Delta_{a}\mathcal{B}(x_{n_{-2}:n_0},u_{n_{-2}:n_{-1}})\\
\label{eq:dela}=&\left(\frac{|k_{n_{-1}}|\eta}{|x_{n_{-2}}|} + \frac{|k_{n_{-2}}|\eta}{|x_{n_{-1}}|}\right)\frac{1}{|k_{n_{-1}}| + |k_{n_{-2}}|}
\end{align}
Now, since by assumption we have that $\min\{|x_{n_{-1}}|,|x_{n_{-2}}|\} \geq p \eta$, we can upper bound equations \eqref{eq:delb}, \eqref{eq:dela} by
\begin{align}
\label{eq:delbapprox}\Delta_{b}\mathcal{B}(x_{n_{-2}:n_0},u_{n_{-2}:n_{-1}})&\leq \frac{1}{p}\frac{2}{|k_{n_{-1}}| + |k_{n_{-2}}|}\\
\label{eq:delaapprox} \Delta_{a}\mathcal{B}(x_{n_{-2}:n_0},u_{n_{-2}:n_{-1}})&\leq \frac{1}{p}
\end{align}
Furthermore, notice that $k_{max}(\mathcal{P}_{n_{-2}}) \geq k_{max}(\mathcal{P}_{n_{-1}})$ so we have
\begin{align}
&&|k_{n_{-1}}| + |k_{n_{-2}}| &\geq \frac{2 k_{max}(\mathcal{P}_{n_{-1}})}{\lambda^*p-1}\\
&\Leftrightarrow & \frac{2}{|k_{n_{-1}}| + |k_{n_{-2}}|} &\leq \frac{\lambda^*p-1}{ k_{max}(\mathcal{P}_{n_{-1}})},
\end{align}
which lets us further upper-bound \eqref{eq:delbapprox} by
\begin{align}
\label{eq:delbapprox2}\Delta_{b}\mathcal{B}(x_{n_{-2}:n_0},u_{n_{-2}:n_{-1}})&\leq \frac{1}{k_{max}(\mathcal{P}_{n_{-1}})}\left(\lambda^*-\frac{1}{p}\right)
\end{align}
Finally, plugging the bounds \eqref{eq:delbapprox2} and \eqref{eq:delaapprox} into equation \eqref{eq:delbdela} gives us the desired result:
\begin{align*}
\lambda(\mathcal{P}_{n_0}) &\leq \frac{k_{max}(\mathcal{P}_{n_{-1}})}{k_{max}(\mathcal{P}_{n_{-1}})}\left(\lambda^*-\frac{1}{p}\right) + \frac{1}{p}\\
\leq \lambda^*
\end{align*}
\end{proof}

With this ability to learn the uncertainty set, we now show how a two-stage controller can lead to a stabilizing (in the BIBO sense) adaptive controller.

\begin{thm}[Passive-Aggressive Learning and Control]\label{thm:PALC}
Let $\mathcal{P}_0$ be an initial (not necessarily strongly stabilizable) uncertainty set, and fix positive constants $\lambda^*,p>0$, s.t. $1-\frac{1}{p}>\lambda^*>\frac{1}{p}$. Consider the system \eqref{eq:dyn} with the following switched control strategy:
\begin{align*}
u_n(x_n)=& \left\{ \begin{array}{ll} u^{PAL}_n(x_n) & \text{If $\lambda(\mathcal{P}_n) > \lambda^*$} \\
u^{ARSF}_n(x_n,\mathcal{P}_n) & \text{If $\lambda(\mathcal{P}_n) \leq \lambda^*$}
\end{array} \right.
\end{align*}
where $u^{PAL}_n$ represents the Passive-Aggressive Learner from \eqref{thm:PALearn} and $u^{ARSF}_n$ is one of the Adaptive Robust State-Feedback controller from section \eqref{sec:ARSF}. Then, it holds:
\begin{enumerate}[(i)]
\item \label{it:bound1} The closed loop system response $\left\| x\right\|_{\infty}$ is bounded as:
\begin{align}
\left\| x\right\|_{\infty} \leq \max_{v \in \left\{ x_0, p\eta  \right\}} |v|\vee Q^{1:2}\left(v,\mathcal{P}_0,u^{PAL}_{0:1}\right)
\end{align}
\item \label{it:bound2} If $\lambda(\mathcal{P}_{n'}) \leq \lambda^*$, then $\forall n \geq n'$, $x_{n}$ satisfies the convergence bounds associated with $u_{ARSF}$ from section \eqref{sec:ARSF}. In particular, it always satisfies the bound
\begin{align*}
|x_n|& \leq \left(\lambda^{*}\right)^{n-n'}|x_{n'}| + \frac{\eta}{1-\lambda^*} 
\end{align*}
\end{enumerate}
\end{thm}
\begin{proof}
\eqref{it:bound2} follows directly from our analysis in Sec.\eqref{sec:ARSF}: Since $\lambda(\mathcal{P}_{n'}) \leq \lambda^* <1-1/p<1$, the parameter uncertainty set is strongly stabilizable and will stay so for all $n\geq n'$. Therefore, the convergence bounds from Sec.\eqref{sec:ARSF} apply.\\
To establish \eqref{it:bound1}, we have to consider different worst-case scenarios. First, assume $\lambda(\mathcal{P}_0) \leq \lambda^*$ holds, then we get the bound $\left\| x\right\|_{\infty} \leq |x_0| \vee p\eta$ from part \eqref{it:bound2}. On the other hand, if $\lambda(\mathcal{P}_0) > \lambda^*$, then the system is being controlled by $u^{PAL}$. Recall that under that regime we have guaranteed that if $x_{n'-1},x_{n'-2} \leq p \eta$, we obtain a strongly stabilizable uncertainty set for $n'$, i.e. $\lambda(\mathcal{P}_{n'}) \leq \lambda^*$. Therefore, for all times after $n'$ we will apply $u^{ARSF}$ and by our previous discussion in \eqref{it:bound2}, we obtain the bound $\left\|x_{n':\infty} \right\|_{\infty} \leq |x_{n'}| \vee p\eta$. It is therefore, left to consider the worst case $\left\|x_{0:n'}\right\|_{\infty}$, where $n'-2$ is the first time that $x$ grows outside the interval $\left[-p\eta,+p\eta \right]$. It is easy to see that we can bound the worst case transient for $\left\|x_{0:n'}\right\|_{\infty}$, by considering the two time-step worst-case transient starting from an initial condition $x_0 \geq p\eta$ and no prior knowledge of the system. Stating that transient bound in terms of the definitions in Sec. \eqref{sec:prob}, we get 
\begin{align*}
\left\|x_{0:n'}\right\|_{\infty} \leq  \max_{v \in \left\{ x_0, p\eta  \right\}} |v|\vee Q^{1:2}\left(v,\mathcal{P}_0,u^{PAL}_{0:1}\right)
\end{align*} 
\end{proof}
The ``passive-aggressive'' nomenclature is chosen because the policy is such that if the system parameters and process noise are stabilizing (i.e., keep the state close to the origin), then we only focus on learning the uncertainty set via the learning control policy \eqref{eq:learnu} -- this is the passive phase of the control policy.  In particular, if for the specific process noise and system parameter realizations no such $n_0$ exists, then by definition we are guaranteed to have $|x_n| \leq C\eta\vee |x_0|$ for all $n$, where $C$ is a constant depending only on $p$, $\lambda^*$ and $k_{max}(\mathcal{P}_n)$ -- this follows by noting from \eqref{cor:Vsym} that $Q^{1:2}\left(v,\mathcal{P}_0,u^{PAL}_{0:1}\right)$ grows sublinear in $|v|$. In contrast, when the noise is such that the state is pushed sufficiently far from the origin, we are able to aggressively decrease the stability margin of the uncertainty set and switch to an ARSF policy.  The result is a stabilizing control scheme that is applicable to arbitrary initial uncertainty sets $\mathcal{P}_0$. 

\section{Simulation Results}
\label{sec:simulations}
In the following we show how the passive-aggressive controller performs under different scenarios. Recall, that our plant is modeled as
\begin{align}
\label{eq:dyn2} x_{n+1} &= a x_n + b u_n + w_n\\
&\left|w_n\right| \leq \eta\,\,\,\begin{bmatrix}a\\b \end{bmatrix} \in \mathcal{P}_0
\end{align}
We set the initial uncertainty set to be
\begin{align}
\mathcal{P}_0 = \left\{\left.\begin{bmatrix} a\\b\end{bmatrix} \right| -3\leq a \leq 3,\,\, 0.1\leq b \leq 3 \right\}
\end{align}
and we pick the true parameters of the system to be
\begin{align}
a_0 &= 2 & b_0 &=0.5.
\end{align}
The $u^{SRSF}$ controller is parametrized with $p=10$ and $\lambda^* = 0.5$.
Notice that this initial uncertainty set is not strongly stabilizable. In what follows, we apply the controller described in Theorem \ref{thm:PALC} for different noise and system parameter realizations: fixed system parameters and adversarial/random noise, adversarial system parameters and noise, and fixed parameters with no noise.

Each of the figures \eqref{fig:rnd}-\eqref{fig:nonoise} show sequences of the state $x_n$ and control action $u_n$ and the maximum and minimum feasible $a$ and $b$ of the current polytope $\mathcal{P}_n$. The right subfigures overlay the area of all polytopes $\mathcal{P}_n$ and display how the uncertainty polytopes $\mathcal{P}_n$ shrink with each iteration. The shade of the polytopes becomes lighter with increasing $n$.

For the simulations with adversarial noise, we choose $w_k = \mathrm{sign}(ax_k+bu_k)$ which is easily seen to be the solution to the inner maximization problem in equation \eqref{thm:minmax} with the surrogate function $V_{RSF}$ to determine the adversarial noise and system parameters.

Notice that $\lambda(\mathcal{P}_n)$ decreases monotonically, and in presence of noise the controller learns to stabilize the system within two time-steps. It is also worth noticing that in presence of no noise, the controller still chooses to perturb the system on purpose to gather information, i.e. an exploration phase naturally emerges to better identify the system parameters before a robustly stabilizing control policy is applied.

\begin{figure}[ht]
\centering
\subcaptionbox{}{\includegraphics[width=0.5\textwidth]{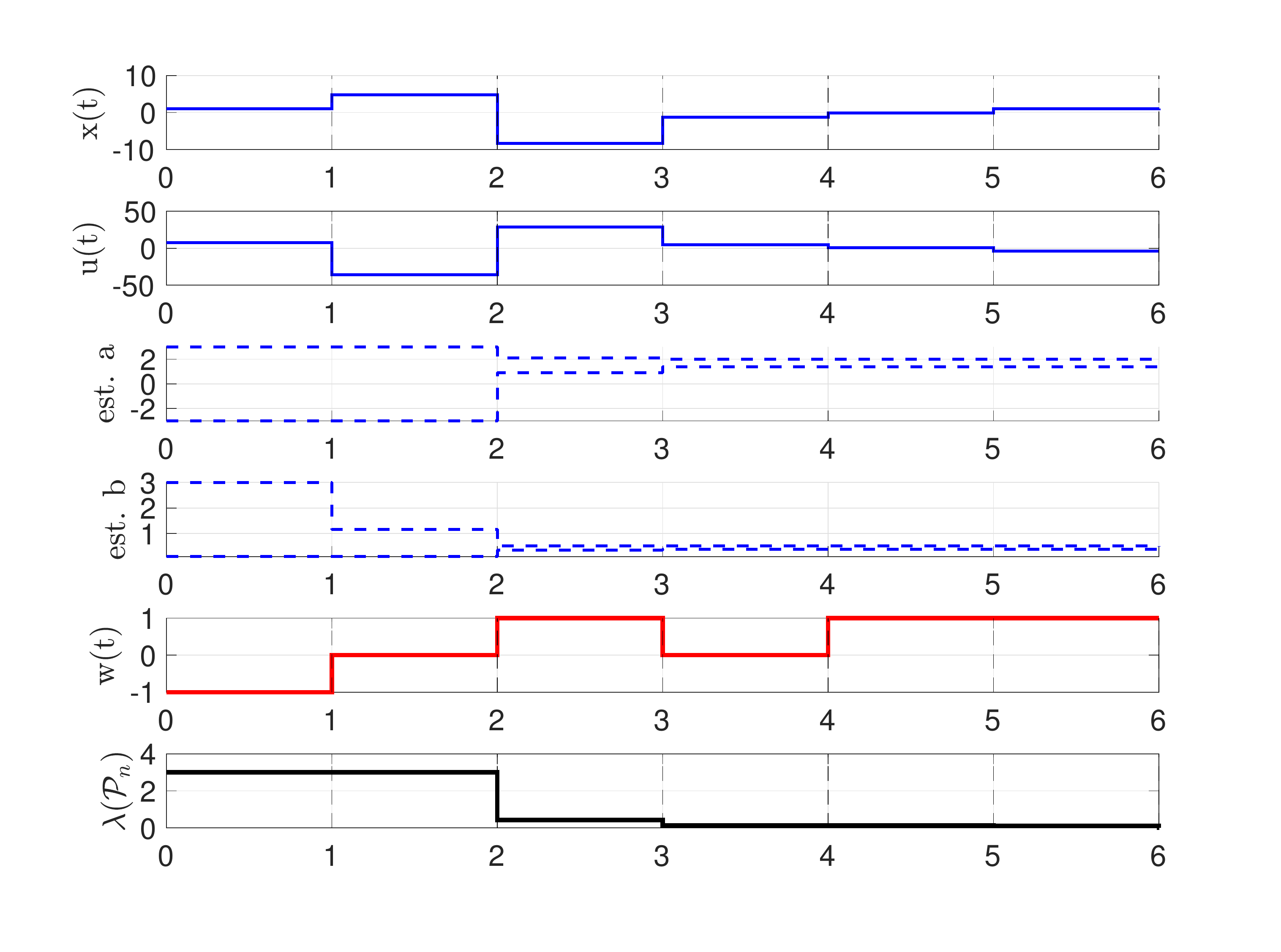}}
\subcaptionbox{}{\includegraphics[width=0.5\textwidth]{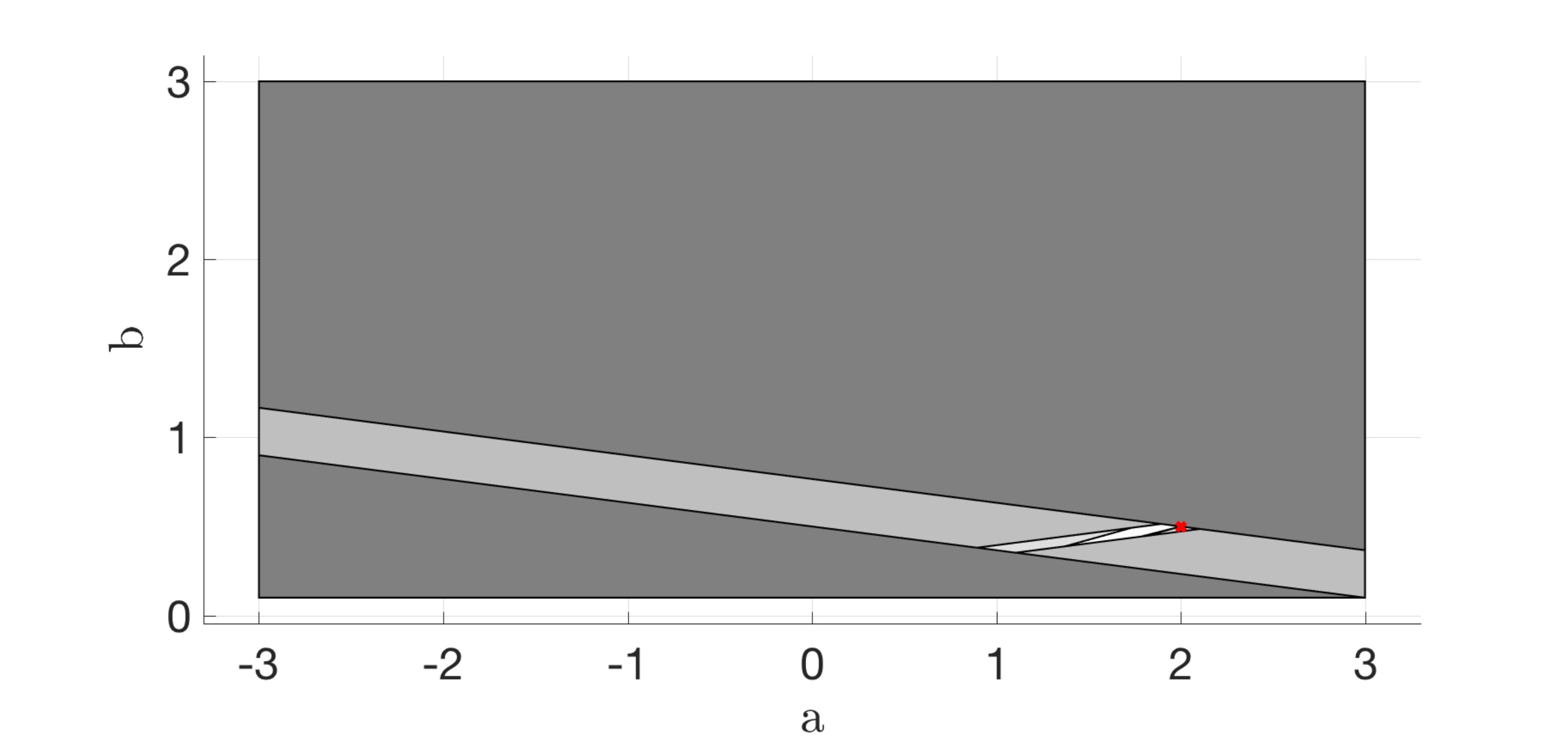}}
\caption{$x_0=1$, $a_0=2$, $b_0=0.5$, $\eta=1$, $\lambda^* = 0.5$, random noise }
\label{fig:rnd}
\end{figure}

\begin{figure}[ht]
\centering
\subcaptionbox{}{\includegraphics[width=0.5\textwidth]{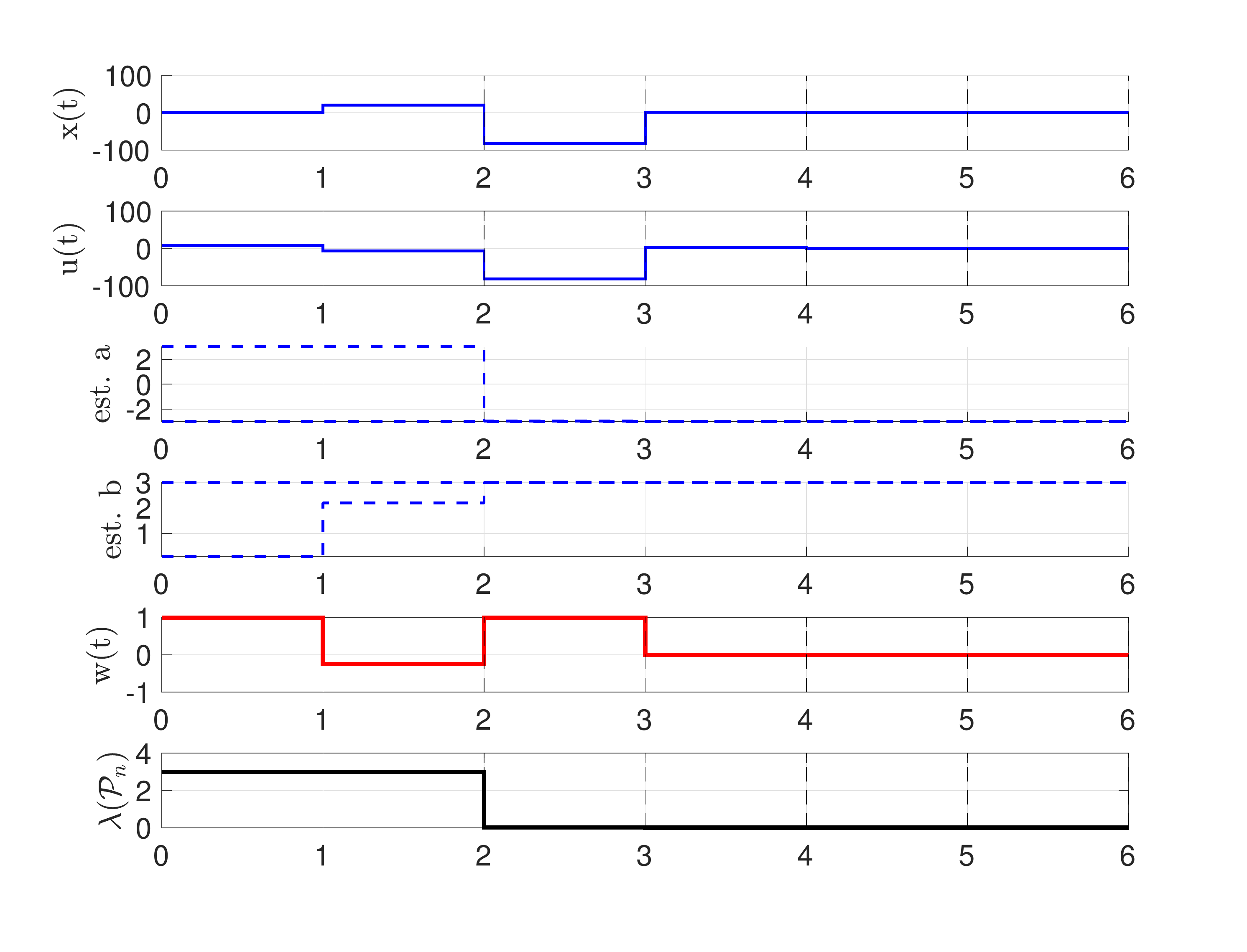}}
\subcaptionbox{}{\includegraphics[width=0.5\textwidth]{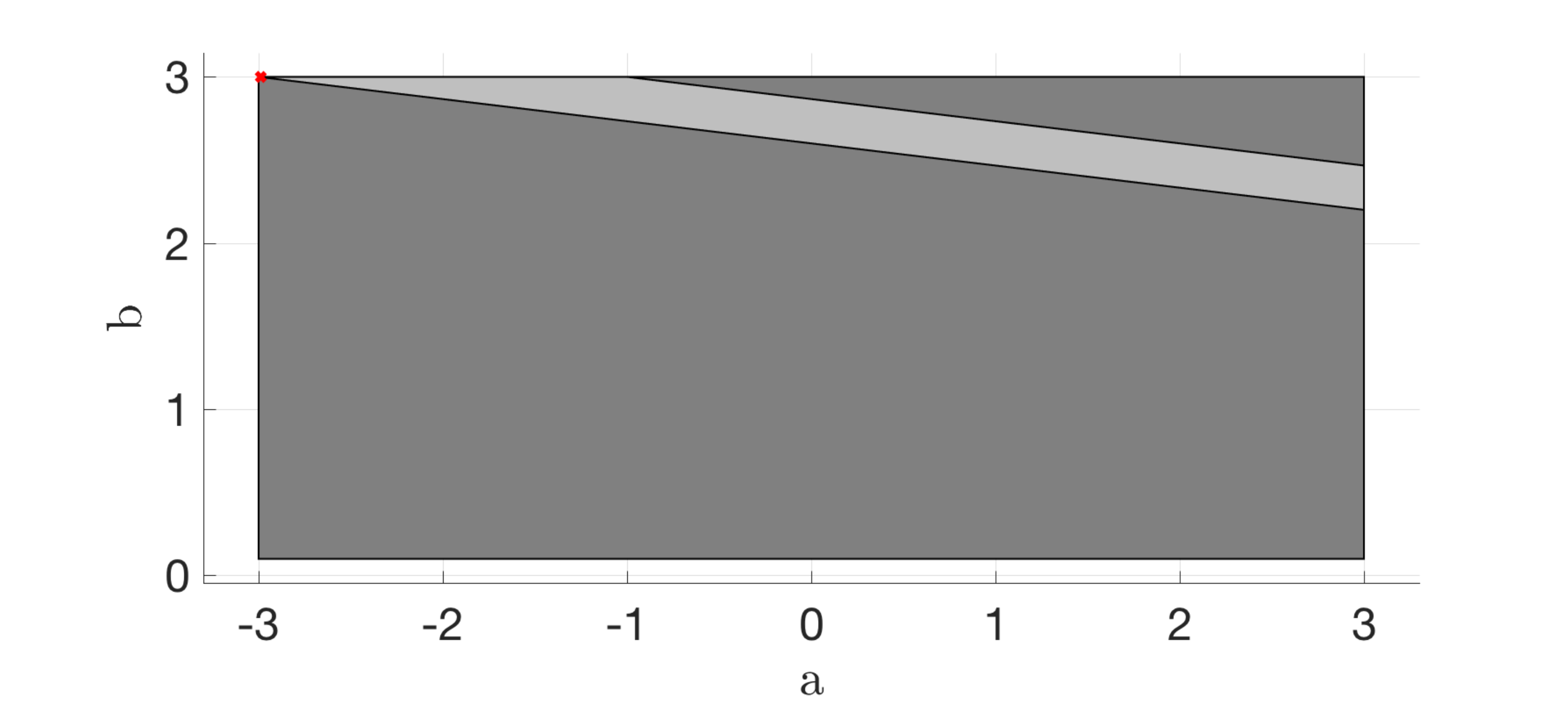}}
\caption{$x_0=1$, $a_0=3$, $b_0=3$, $\eta=1$, $\lambda^* = 0.5$, adversarial noise and system }
\label{fig:advnoisesys}
\end{figure}

\begin{figure}[ht]
\centering
\subcaptionbox{}{\includegraphics[width=0.5\textwidth]{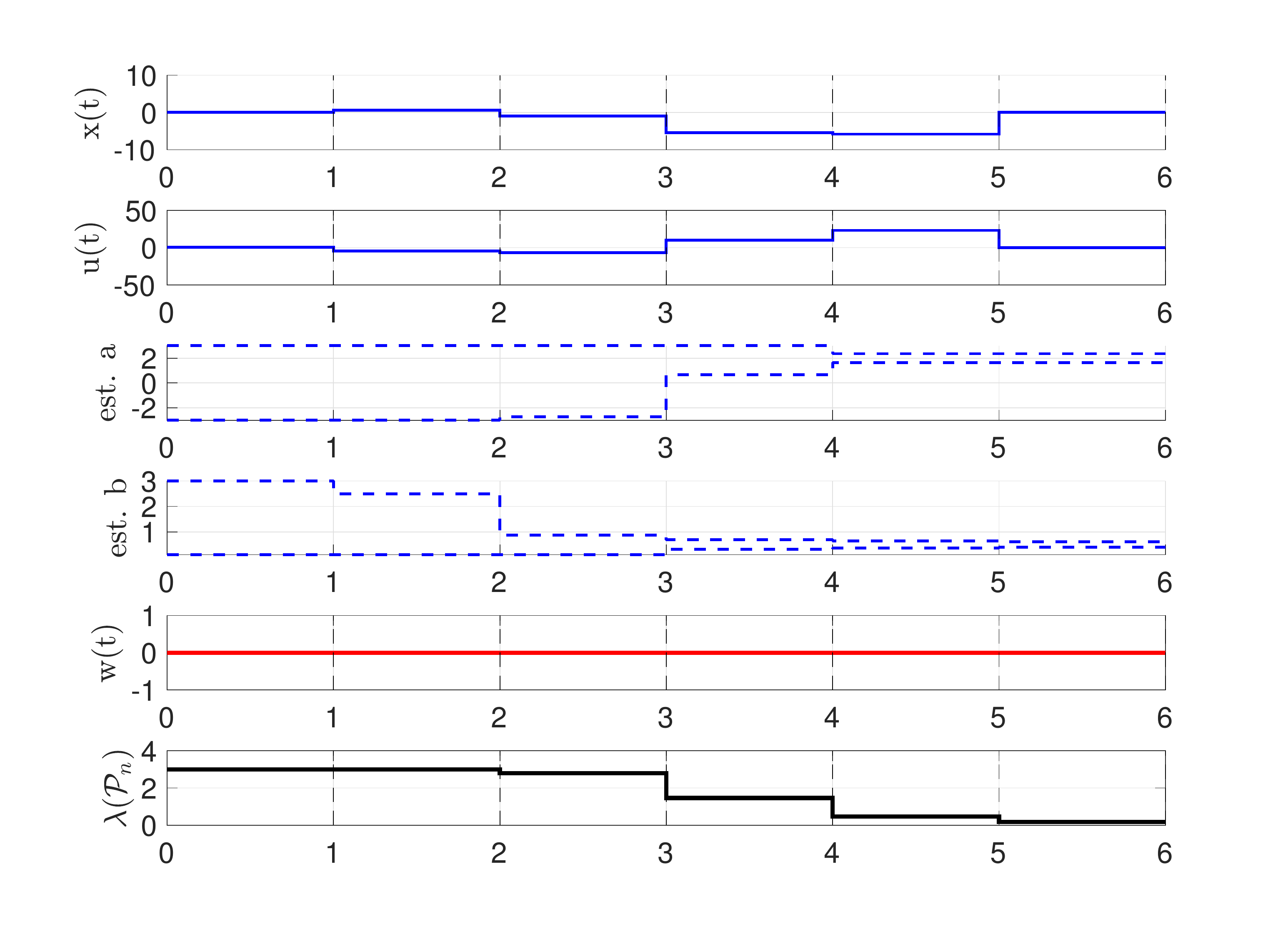}}
\subcaptionbox{}{\includegraphics[width=0.5\textwidth]{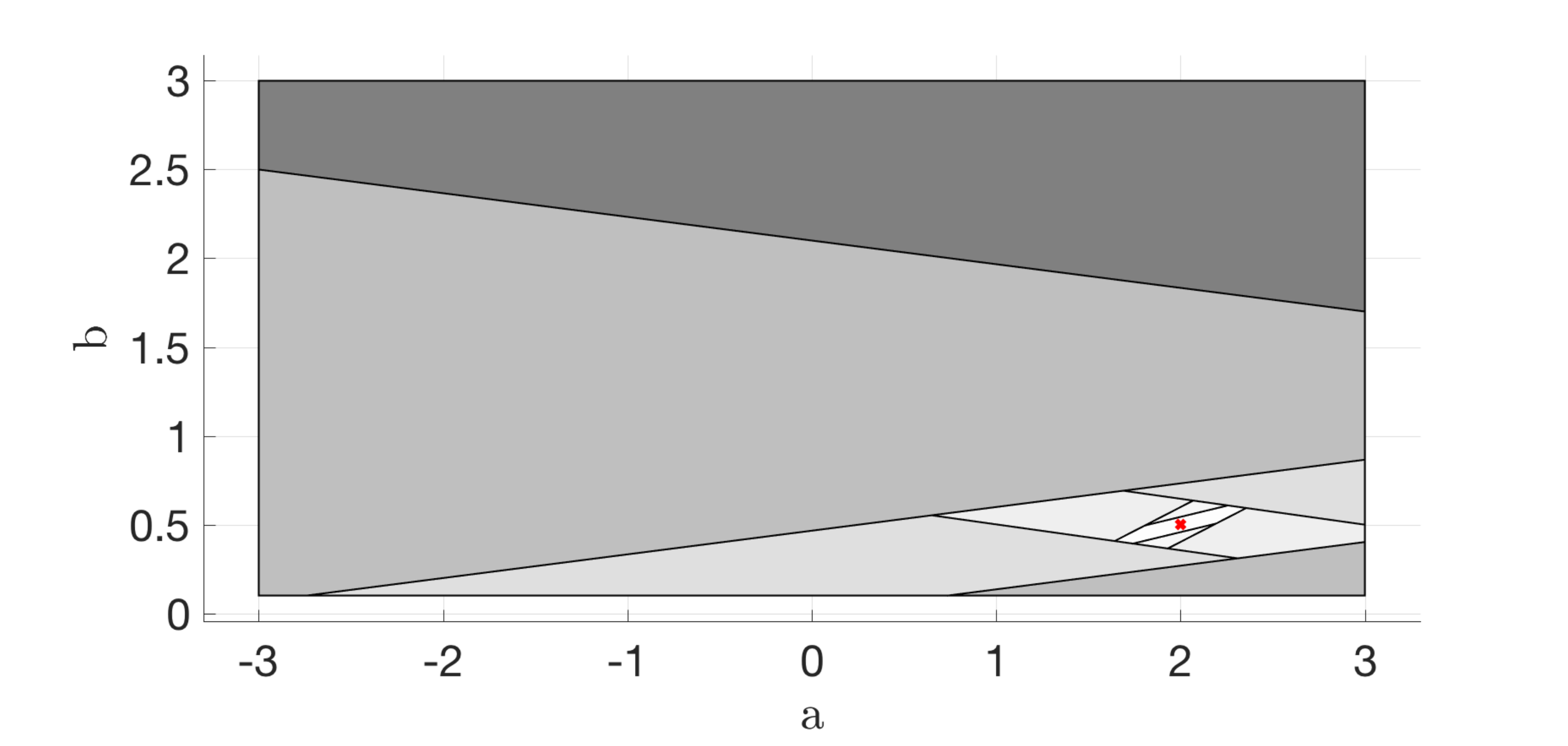}}
\caption{$x_0=0.1$, $a_0=2$, $b_0=0.5$, $\eta=1$, $\lambda^* = 0.5$, no noise }
\label{fig:nonoise}
\end{figure}

\section{Conclusions and Future Work}
In this paper we defined and analyzed the passive-aggressive learning and control strategy for scalar systems with bounded but adversarial process noise and parametric uncertainty.  We showed that for strongly stabilizable initial uncertainty sets, sharp bounds on the state-deviation can be obtained using an ARSF control policy.  We then extended these results to the general setting by proposing a two-stage controller: the first stage seeks to passively learn the system so long as the state remains sufficiently close to the origin.  However, if the process and system noise are such that the state is pushed sufficiently far from the origin, the controller is able to aggressively reduce the uncertainty set to one that is strongly stabilizable, thus allowing for either the weakly or strongly ARSF policies to be applied.  Future work will look to actively inject noise into the passive stage of the aforementioned two-stage control policy to expedite the learning process, as well as characterize sharp regret bounds on the proposed policy.  Of additional interest is the extension of the proposed methods to the vector valued setting.

\appendices
\section{Stability Margin Bounds}\label{sec:box}
\subsection{Box-shaped Uncertainty Sets}
\begin{lem}\label{lem:box}
Let $\mathcal{B}$ be a controllable boxed uncertainty 
\begin{align*}
\mathcal{B} &= \left\{\left.\begin{bmatrix} a\\b \end{bmatrix} \right| \begin{array}{c} l_a \leq a \leq u_a\\ 0<l_b \leq b \leq u_b \end{array} \right\} 
\end{align*}
and define $\Delta_b \mathcal{B}$, $\Delta_a \mathcal{B}$, $a_{av}$, $b_{av}$, $k_{av}$ as
\begin{align*}
a_{av} = (u_a+l_a)/2 && b_{av} = (u_b+l_b)/2\\
\Delta_a \mathcal{B} = (u_a-l_a)/2 && \Delta_b \mathcal{B} = (u_b-l_b)/2\\
&k_{av} = -\frac{a_{av}}{b_{av}}
\end{align*}
where $a_{av}$, $b_{av}$ is the average system of the box and $k_{av}$ the corresponding deadbeat feedback. Then the stability margin of $\mathcal{B}$ can be computed as
\begin{align*}
\lambda(\mathcal{B}) = |k_{av}|\Delta_b \mathcal{B}+\Delta_a \mathcal{B}
\end{align*}
\end{lem}
\begin{proof}
The proof is omitted but follows by solving the following optimization problem 
\begin{align*}
\min\limits_{k} \max\limits_{\begin{array}{c} l_a \leq a \leq u_a\\ l_b \leq b \leq u_b \end{array}} \left| a+bk\right|
\end{align*}
\end{proof}
 \begin{figure}[ht]
 \centering
 \includegraphics[width=0.5\textwidth]{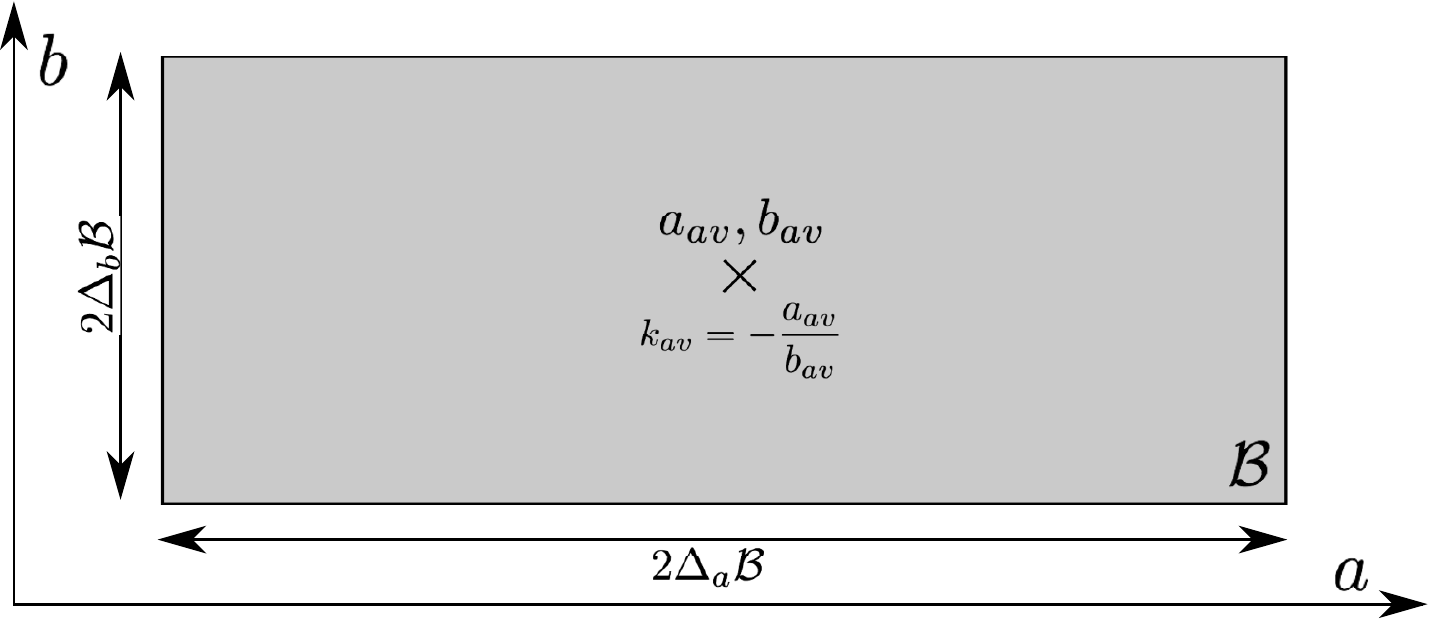}
 \caption{Example of a box uncertainty}
 \label{fig:box}
 \end{figure}

\subsection{Approximation for $S(x_{i:i+2},u_{i:i+1})$ in Passive-Aggressive Learning}
Consider applying $u_1 = k_1x_1$ and $u_2= -k_2 x_2$ as a feedback controller with $k_1>0$, $k_2>0$. Then the resulting uncertainty set $S(x_{i:i+2},u_{i:i+1})$ resembles a parallelogram as shown in Fig.\eqref{fig:parallelogram}. An approximation of the stability margin $\lambda(S(x_{i:i+2},u_{i:i+1}))$ is the stability margin of its outer-bounding box, i.e. $\lambda(\mathcal{B}(x_{i:i+2},u_{i:i+1}))$. Using the notation in Fig.\eqref{fig:parallelogram} and Lem.\eqref{lem:box}, the approximation can be computed as
\begin{align*}
&\lambda(\mathcal{B}(x_{i:i+2},u_{i:i+1}))\\
= &k_{av}(\mathcal{B}(x_{i:i+2},u_{i:i+1}))(x+y) + k_1x +k_2 y
\end{align*}
From simple geometry, notice that the shaded area $A$ can be computed in three ways:
\begin{align}
A&= h_1 \sqrt{1+k^2_1}x = 2\eta\frac{x}{|x_1|}\\
 &= h_2 \sqrt{1+k^2_2}y = 2\eta\frac{y}{|x_2|}\\
 &= (x+y)(k_1x+k_2y)-k_1x^2-k_2y^2
\end{align}
We can use these equations to solve for $x$ and $y$ and finally obtain:
\begin{align}
\notag&\lambda(\mathcal{B}(x_{i:i+2},u_{i:i+1}))\\
\label{eq:k_approx}=&k_{av}\left(\mathcal{B}(\dots)\right)\frac{\frac{\eta}{|x_2|} + \frac{\eta}{|x_1|}}{k_2 + k_1} + \frac{k_2\frac{\eta}{|x_1|}+k_1\frac{\eta}{|x_2|}}{k_1 +k_2}
\end{align}
\begin{figure}[ht]
\centering
\includegraphics[width=0.5\textwidth]{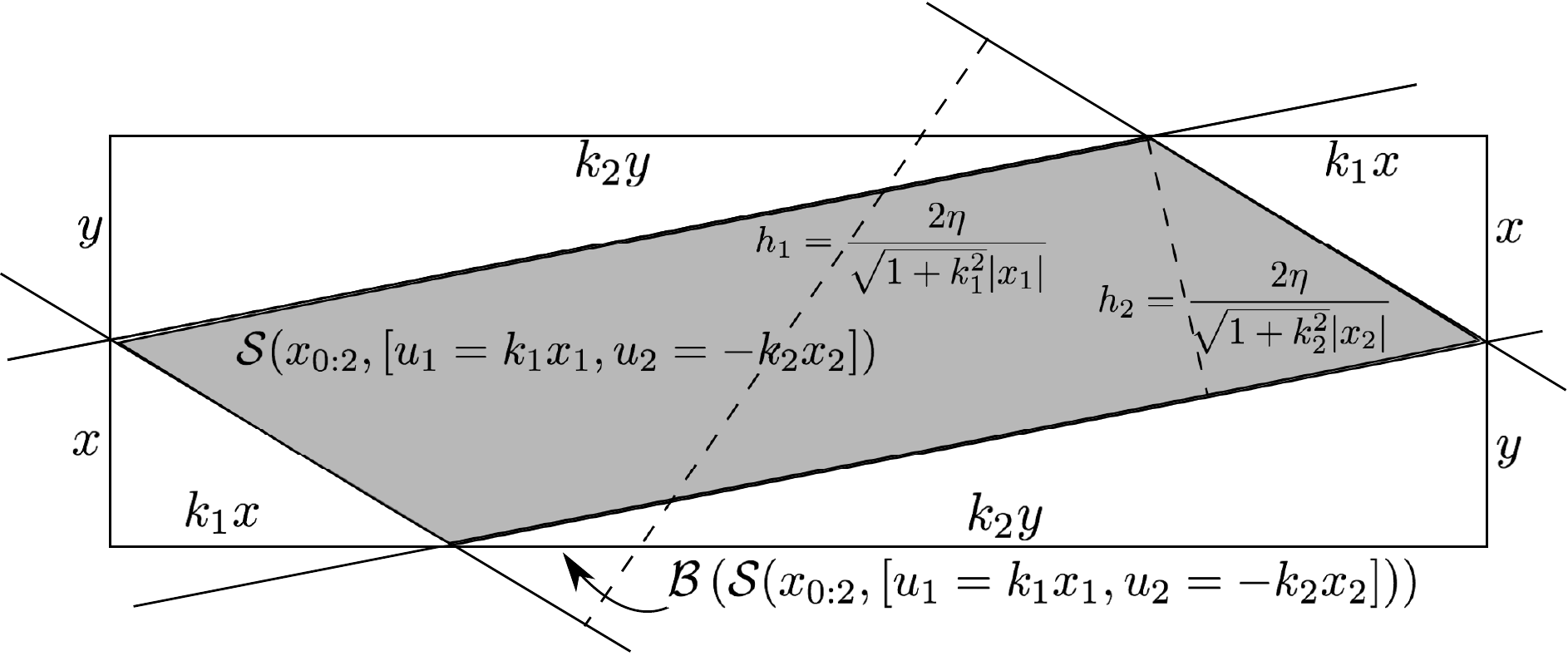}
\caption{Example uncertainty set after two timesteps of the passive-aggressive learner}
\label{fig:parallelogram}
\end{figure}
\bibliographystyle{IEEEtran}
\bibliography{IEEEabrv,conbib}
\end{document}